\documentclass[12pt,reqno]{amsart}
\usepackage{amsmath, amsfonts, amssymb, amsthm, hyperref,cite}
\usepackage{bm}
\usepackage{mathrsfs}
\allowdisplaybreaks[4]
\def\l{\left}
\def\r{\right}
\def\bg{\bigg}
\def\({\bg(}
\def\){\bg)}
\def\t{\text}
\def\f{\frac}

\def\eq{\equiv}

\def\<{\langle}
\def\>{\rangle}
\def\1{{\bf 1}}

\theoremstyle{plain}
\newtheorem{theorem}{Theorem}
\newtheorem{conjecture}{Conjecture}

\newtheorem{lemma}{Lemma}

\theoremstyle{definition}

\theoremstyle{remark}

\numberwithin{equation}{section}

\begin{document}
\title[On a conjectural supercongruence involving the dual sequence $s_n(x)$]{On a conjectural supercongruence involving \\ the dual sequence $s_n(x)$}

\author[Chen Wang]{Chen Wang$^{\ast}$}
\address{Department of Applied Mathematics, Nanjing Forestry University, Nanjing 210037, People's Republic of China}
\email{cwang@smail.nju.edu.cn}

\author[Sheng-Jie Wang]{Sheng-Jie Wang}
\address{Department of Applied Mathematics, Nanjing Forestry University, Nanjing 210037, People's Republic of China}
\email{sjw@njfu.edu.cn}

\begin{abstract}
In 2017, motivated by a supercongruence conjectured by Kimoto and Wakayama and confirmed by Long, Osburn and Swisher, Z.-W. Sun introduced the sequence of polynomials:
$$
s_n(x)=\sum_{k=0}^n\binom{n}{k}\binom{x}{k}\binom{x+k}{k}=\sum_{k=0}^n\binom{n}{k}(-1)^k\binom{x}{k}\binom{-1-x}{k}
$$
and investigated its congruence properties. In particular, Z.-W. Sun conjectured that for any prime $p>3$ and $p$-adic integer $x\neq-1/2$ one has
\begin{equation*}
\sum_{n=0}^{p-1}s_n(x)^2\equiv (-1)^{\langle x\rangle_p}\frac{p+2(x-\langle x\rangle_p)}{2x+1}\pmod{p^3},
\end{equation*}
where $\langle x\rangle_p$ denotes the least nonnegative residue of $x$ modulo $p$. In this paper, we confirm this conjecture.
\end{abstract}
\thanks{$^{\ast}$Corresponding author}

\maketitle

\section{Introduction}

For a sequence of numbers $a_0,a_1,a_2,\ldots$, its dual sequence $a_0^*,a_1^*,a_2^*,\ldots$ is defined by
$$
a_n^*=\sum_{k=0}^n\binom{n}{k}(-1)^ka_k\quad (n=0,1,2,\ldots).
$$
By the binomial inversion formula, we have $(a_n^*)^*=a_n$ for all $n=0,1,2,\ldots$.

Let $p$ be an odd prime. In 2013, Z.-W. Sun \cite[Theorem 2.4]{SunZW2013} proved that for $p>3$,
\begin{align*}
\sum_{k=0}^{p-1}\binom{-1/3}{k}\binom{-2/3}{k}a_k&\eq\l(\f{p}{3}\r)\sum_{k=0}^{p-1}\binom{-1/3}{k}\binom{-2/3}{k}a_k^*\pmod{p^2},\\
\sum_{k=0}^{p-1}\binom{-1/4}{k}\binom{-3/4}{k}a_k&\eq\l(\f{-2}{p}\r)\sum_{k=0}^{p-1}\binom{-1/4}{k}\binom{-3/4}{k}a_k^*\pmod{p^2},\\
\sum_{k=0}^{p-1}\binom{-1/6}{k}\binom{-5/6}{k}a_k&\eq\l(\f{-1}{p}\r)\sum_{k=0}^{p-1}\binom{-1/6}{k}\binom{-5/6}{k}a_k^*\pmod{p^2},
\end{align*}
where $a_0,a_1,a_2,\ldots$ are $p$-adic integers and $(-)$ stands for the Legendre symbol. In 2014, Z.-H. Sun \cite[Theorem 2.4]{SunZH2014} gave a parametric extension of the above supercongruences: For any $p$-adic integer $x$ one has
$$
\sum_{k=0}^{p-1}\binom{x}{k}\binom{-1-x}{k}a_k\eq (-1)^{\<x\>_p}\sum_{k=0}^{p-1}\binom{x}{k}\binom{-1-x}{k}a_k^*\pmod{p^2},
$$
where $\<x\>_p$ denotes the least nonnegative residue of $x$ modulo $p$.

It seems that if any sequence of $p$-adic integers possesses a good congruence property, then its dual sequence does too. Another evidence is concerned with the square of a dual sequence. In 2003, Rodriguez-Villegas \cite{RV} conjectured that for any prime $p>3$,
$$
\sum_{k=0}^{p-1}\binom{-1/2}{k}^2\eq\l(\f{-1}{p}\r)\pmod{p^2},
$$
which was later confirmed by Mortenson \cite{Mortenson} and generalized by Z.-W. Sun \cite{SunZW2012} in the following form:
$$
\sum_{k=0}^{p-1}\binom{-1/2}{k}^2\eq\l(\f{-1}{p}\r)-p^2E_{p-3}\pmod{p^3},
$$
where $E_{p-3}$ is the $(p-3)$th Euler number. In 2006, when studied special values of spectral zeta functions, Kimoto and Wakayama \cite{KW} introduced the Ap\'ery-like numbers:
$$
\tilde{J}_2(n)=\sum_{k=0}^n\binom{n}{k}(-1)^k\binom{-1/2}{k}^2,
$$
which is actually the dual sequence of $\binom{-1/2}{n}^2$. It is natural to ask that does $\tilde{J}_2(n)$ have any good congruence property? The answer is positive. In fact, Kimoto and Wakayama \cite{KW} conjectured that
\begin{equation}\label{KWcon}
\sum_{n=0}^{p-1}\tilde{J}_2(n)^2\eq\l(\f{-1}{p}\r)\pmod{p^3}
\end{equation}
for any odd prime $p$. This conjecture was confirmed by Long, Osburn and Swisher \cite{LOS} in 2016.

In 2017, Z.-W. Sun \cite{SunZW2017} investigated some sophisticated supercongruences involving dual sequences. Motivated by the work of Kimoto and Wakayama, Z.-W. Sun defined the sequence of polynomials:
$$
s_n(x)=\sum_{k=0}^n\binom{n}{k}\binom{x}{k}\binom{x+k}{k}=\sum_{k=0}^n\binom{n}{k}(-1)^k\binom{x}{k}\binom{-1-x}{k}.
$$
Clearly, $\tilde{J}_2(n)=s_n(-1/2)$. Similarly to \eqref{KWcon}, Z.-W. Sun studied supercongruences for $\sum_{n=0}^{p-1}s_n(x)^2$, where $p>3$ is a prime and $x$ is a $p$-adic integer. In particular, he proved that if $x\not\eq-1/2\pmod{p}$, then we have
\begin{equation}\label{suncon}
\sum_{n=0}^{p-1}s_n(x)^2\eq (-1)^{\<x\>_p}\f{p+2(x-\<x\>_p)}{2x+1}\pmod{p^2}.
\end{equation}
Meanwhile, he further posed the following conjecture.
\begin{conjecture}\label{sunconj}
For any prime $p>3$ and $p$-adic integer $x\neq-1/2$, we have
\begin{equation}\label{sunconjeq}
\sum_{n=0}^{p-1}s_n(x)^2\eq (-1)^{\<x\>_p}\f{p+2(x-\<x\>_p)}{2x+1}\pmod{p^3}.
\end{equation}
\end{conjecture}

Our main purpose is to confirm this conjecture.

\begin{theorem}\label{mainth}
Conjecture \eqref{sunconj} is true.
\end{theorem}

It is worth mentioning that when $x=-1/2,-1/3,-1/4,-1/6$, Z.-W. Sun \cite{SunZW2017} even conjectured supercongruences for $\sum_{n=0}^{p-1}s_n(x)^2$ in the modulus $p^4$ cases. Precisely, he conjectured that for any prime $p>3$,
\begin{align}
\sum_{k=0}^{p-1}s_k\l(-\f12\r)^2&\eq\l(\f{-1}{p}\r)(1-7p^3B_{p-3})\pmod{p^4},\label{suncon1}\\
\sum_{k=0}^{p-1}s_k\l(-\f14\r)^2&\eq\l(\f{2}{p}\r)p-26\l(\f{-2}{p}\r)p^3E_{p-3}\pmod{p^4},\label{suncon2}\\
\sum_{k=0}^{p-1}s_k\l(-\f13\r)^2&\eq p-\f{14}{3}\l(\f{p}{3}\r)p^3B_{p-2}\l(\f13\r)\pmod{p^4},\label{suncon3}\\
\sum_{k=0}^{p-1}s_k\l(-\f16\r)^2&\eq\l(\f{3}{p}\r)p-\f{155}{12}\l(\f{-1}{p}\r)p^3B_{p-2}\l(\f13\r)\pmod{p^4},\label{suncon4}
\end{align}
where $B_{p-3}$ is the $(p-3)$th Bernoulli number and $B_{p-2}(y)$ stands for the $(p-2)$th Bernoulli polynomial. \eqref{suncon1} is actually an extension of \eqref{KWcon} and was confirmed by Liu \cite{Liu2018} in 2018. By Theorem \ref{mainth}, we immediately obtained \eqref{suncon2}--\eqref{suncon4} in the modulus $p^3$ cases. For more supercongruences involving the dual sequence $s_n(x)$, we refer the reader to \cite{Guo,Liu2017,Mao,SunZH2022,WangZhong}.

The proof of Theorem \ref{mainth} has two key ingredients. One involves the $p$-adic expansion of $\binom{x}{k}\binom{x+k}{k}$ for any $p$-adic integer $x$ with $\<x\>_p\leq (p-1)/2$ and $0\leq k\leq p-1$, and the other includes many complicated combinatorial identities that reveal hidden symmetries behind the original multiple sum $\sum_{n=0}^{p-1}s_n(x)^2$. The paper is organized as follows. In the next section, we shall present some preliminary results. In Section 3, we shall prove Theorem \ref{mainth} in three cases.

\section{Preliminary results}
Let $p$ be an odd prime. Throughout the proof, for any $p$-adic integer $x$, we always write $m=\<x\>_p$ and $t=(x-m)/p$. For any nonnegative integer $n$ and positive integer $m$, the $n$th harmonic number of order $m$ is defined by
$$
H_n^{(m)}:=\sum_{k=1}^n\f{1}{k^m}.
$$
For convenience, we usually use $H_n$ to denote $H_n^{(1)}$.

In order to prove Theorem \ref{mainth}, we need the following lemmas.
\begin{lemma}\label{xkx+1k}
For any odd prime $p$ and $p$-adic integer $x$ with $m\leq(p-1)/2$, we have
\begin{align*}
\binom{x}{k}\binom{x+k}{k}\eq\begin{cases}\binom{m}{k}\binom{m+k}{k}(1+ptH_{m+k}-ptH_{m-k})\pmod{p^2},\quad&\t{if}\ 0\leq k\leq m,\vspace{8pt}\\ \f{(-1)^mp^2t(t+1)}{k(k-m)\binom{m}{p-k}\binom{k}{m}}\pmod{p^3},\quad&\t{if}\ p-m\leq k\leq p-1.\end{cases}
\end{align*}

Moreover, if $m<(p-1)/2$ and $m+1\leq k\leq p-1-m$, then we have
\begin{align*}
\binom{x}{k}\binom{x+k}{k}\eq\f{(-1)^{m+k+1}pt\binom{m+k}{k}}{(k-m)\binom{k}{m}}(1+ptH_{m+k}-ptH_{k-m-1})\pmod{p^3}.
\end{align*}
\end{lemma}

\begin{proof}
Set
$$
F(y):=\binom{y}{k}\binom{y+k}{k}.
$$
It is easy to check that
$$
\f{F(y+1)}{F(y)}=\f{y+1+k}{y+1-k}.
$$
Then we have
\begin{align*}
F(x)=F(pt)\cdot \f{F(m+pt)}{F(pt)}=F(pt)\prod_{y=pt}^{m+pt-1}\f{F(y+1)}{F(y)}=F(pt)\prod_{y=pt}^{m+pt-1}\f{y+1+k}{y+1-k}.
\end{align*}
Note that for $k\in\{0,1,2,\ldots,p-1\}$,
\begin{equation}\label{Fpt}
F(pt)=\binom{pt}{k}\binom{pt+k}{k}=\f{pt}{pt-k}\prod_{i=1}^k\f{p^2t^2-i^2}{i^2}=\f{(-1)^kpt}{pt-k}\prod_{i=1}^{k}\l(1-\f{p^2t^2}{i^2}\r).
\end{equation}
Furthermore,
\begin{equation}\label{Fy}
\prod_{y=pt}^{m+pt-1}\f{y+1+k}{y+1-k}=(-1)^m\f{\prod_{i=k+1}^{k+m}(i+pt)}{\prod_{i=k-m}^{k-1}(i-pt)}.
\end{equation}

{\it Case 1}. $0\leq k\leq m$.

Clearly, the product $\prod_{i=k-m}^k(i-pt)$ contains the factor $-pt$. Combining \eqref{Fpt} and \eqref{Fy}, we obtain
\begin{align*}
F(x)&=(-1)^{m+k+1}\prod_{i=1}^{k}\l(1-\f{p^2t^2}{i^2}\r)\f{\prod_{i=k+1}^{k+m}(i+pt)}{\prod_{\substack{i=k-m\\ i\neq0}}^{k}(i-pt)}\\
&=\prod_{i=1}^{k}\l(1-\f{p^2t^2}{i^2}\r)\f{\prod_{i=k+1}^{k+m}(i+pt)}{(\prod_{i=1}^{m-k}(i+pt))(\prod_{i=1}^{k}(i-pt))}\\
&\eq \f{(\prod_{i=k+1}^{k+m}i)(1+ptH_{k+m}-ptH_k)}{(\prod_{i=1}^{m-k}i)(\prod_{i=1}^{k}i)(1+ptH_{m-k})(1-ptH_k)}\\
&\eq\binom{m}{k}\binom{m+k}{k}(1+ptH_{m+k}-ptH_{m-k})\pmod{p^2}.
\end{align*}

{\it Case 2}. $p-m\leq k\leq p-1$.

Now, the product $\prod_{i=k+1}^{k+m}(i+pt)$ contains the factor $p(t+1)$. Therefore, in view of \eqref{Fpt} and \eqref{Fy}, we deduce that
\begin{align*}
F(x)&=(-1)^{m+k+1}p^2t(t+1)\prod_{i=1}^{k}\l(1-\f{p^2t^2}{i^2}\r)\f{\prod_{\substack{i=k+1\\ i\neq p}}^{k+m}(i+pt)}{\prod_{i=k-m}^{k}(i-pt)}\\
&\eq(-1)^{m+k+1}p^2t(t+1)\f{(\prod_{i=k+1}^{p-1}i)(\prod_{i=1}^{k+m-p}i)}{\prod_{i=k-m}^{k}i}\\
&\eq\f{(-1)^mp^2t(t+1)}{k(k-m)\binom{m}{p-k}\binom{k}{m}}\pmod{p^3}.
\end{align*}

{\it Case 3}. $m<(p-1)/2$ and $m+1\leq k\leq p-1-m$.

In this case, both $\prod_{i=k+1}^{k+m}(i+pt)$ and $\prod_{i=k-m}^{k}(i+pt)$ are coprime with $p$. By \eqref{Fpt} and \eqref{Fy}, we have
\begin{align*}
F(x)&=(-1)^{m+k+1}pt\l(1-\f{p^2t^2}{i^2}\r)\f{\prod_{i=k+1}^{k+m}(i+pt)}{\prod_{i=k-m}^{k}(i-pt)}\\
&\eq(-1)^{m+k+1}pt\f{(\prod_{i=k+1}^{k+m}i)(1+ptH_{m+k}-ptH_k)}{(\prod_{i=k-m}^{k}i)(1-ptH_k+ptH_{k-m-1})}\\
&\eq\f{(-1)^{m+k+1}pt\binom{m+k}{k}}{(k-m)\binom{k}{m}}(1+ptH_{m+k}-ptH_{k-m-1})\pmod{p^3}.
\end{align*}

In view of the above, we are done.
\end{proof}

\begin{lemma}[Liu {\cite[(2.2)]{Liu2018}}]\label{identity1}
For nonnegative integers $j$ and $k$, we have
$$
\sum_{n=0}^{j+k}\f{(-1)^n}{n+1}\binom{n}{j}\binom{j}{n-k}=\f{(-1)^{j+k}}{(j+k+1)\binom{j+k}{j}}.
$$
\end{lemma}

\begin{lemma}[Liu {\cite[(2.4)]{Liu2018}}]\label{identity3}
For nonnegative integers $M$ and $k$, we have
$$
\sum_{j=0}^M\f{(-1)^j\binom{M}{j}\binom{M+j}{j}}{(j+k+1)\binom{j+k}{j}}=\f{\binom{k}{M}}{(k+1)\binom{M+k+1}{M}}.
$$
\end{lemma}

\begin{lemma}[Liu {\cite[Lemma 2.2]{Liu2018}}]\label{identity2}
For any nonnegative integer $M$, we have
\begin{align*}
&\sum_{j=0}^M\sum_{k=0}^M\sum_{n=0}^{j+k}\f{1}{n+1}\binom{M}{j}\binom{M+j}{j}\binom{M}{k}\binom{M+k}{k}\binom{n}{j}\binom{j}{n-k}\binom{y-1}{k}\\
&\quad=\f{(-1)^M}{2M+1}(1-4y^2H_M^{(2)})+\mathcal{O}(y^4).
\end{align*}
\end{lemma}

\begin{lemma}\label{sigma1}
For any odd prime $p$ and $p$-adic integer $x$ with $m<(p-1)/2$, we have
\begin{align*}
&p\sum_{j=0}^m\sum_{k=0}^m\binom{x}{j}\binom{x+j}{j}\binom{x}{k}\binom{x+k}{k}\sum_{n=0}^{j+k}\f{1}{n+1}\binom{n}{j}\binom{j}{n-k}\binom{p-1}{n}\\
&\quad\eq \f{p(-1)^m}{2m+1}+\f{2p^2t(-1)^m}{2m+1}H_{2m}\pmod{p^3}.
\end{align*}
\end{lemma}

\begin{proof}
Since $0\leq j,k\leq m<(p-1)/2$, we have $1\leq n+1\leq j+k+1\leq p-1$. By Lemma \ref{xkx+1k},
\begin{align}\label{sigma1decom}
&p\sum_{j=0}^m\sum_{k=0}^m\binom{x}{j}\binom{x+j}{j}\binom{x}{k}\binom{x+k}{k}\sum_{n=0}^{j+k}\f{1}{n+1}\binom{n}{j}\binom{j}{n-k}\binom{p-1}{n}\notag\\
&\quad\eq p\sum_{j=0}^m\sum_{k=0}^m\binom{m}{j}\binom{m+j}{j}\binom{m}{k}\binom{m+k}{k}(1+pt(H_{m+j}+H_{m+k}-H_{m-j}-H_{m-k}))\notag\\
&\qquad\times\sum_{n=0}^{j+k}\f{1}{n+1}\binom{n}{j}\binom{j}{n-k}\binom{p-1}{n}\notag\\
&\quad \eq p\sum_{j=0}^m\sum_{k=0}^m\binom{m}{j}\binom{m+j}{j}\binom{m}{k}\binom{m+k}{k}(1+2p t(H_{m+k}-H_{m-k}))\notag\\
&\qquad\times\sum_{n=0}^{j+k}\f{1}{n+1}\binom{n}{j}\binom{j}{n-k}\binom{p-1}{n}\pmod{p^3},
\end{align}
where in the last step we have used the symmetry of $j$ and $k$. By Lemma \ref{identity2} with $M=m,\ y=p$, we obtain
\begin{equation}\label{sigma1decom1}
p\sum_{j=0}^m\sum_{k=0}^m\binom{m}{j}\binom{m+j}{j}\binom{m}{k}\binom{m+k}{k}\sum_{n=0}^{j+k}\f{1}{n+1}\binom{n}{j}\binom{j}{n-k}\binom{p-1}{n}\eq \f{p(-1)^m}{2m+1}\pmod{p^3}.
\end{equation}
Moreover, with the help of Lemmas \ref{identity1} and \ref{identity3}, we get
\begin{align}\label{sigma1decom2}
&2p^2t\sum_{j=0}^m\sum_{k=0}^m\binom{m}{j}\binom{m+j}{j}\binom{m}{k}\binom{m+k}{k}(H_{m+k}-H_{m-k})\sum_{n=0}^{j+k}\f{1}{n+1}\binom{n}{j}\binom{j}{n-k}\binom{p-1}{n}\notag\\
&\quad\eq 2p^2t\sum_{j=0}^m\sum_{k=0}^m\binom{m}{j}\binom{m+j}{j}\binom{m}{k}\binom{m+k}{k}(H_{m+k}-H_{m-k})\sum_{n=0}^{j+k}\f{(-1)^n}{n+1}\binom{n}{j}\binom{j}{n-k}\notag\\
&\quad=2p^2t\sum_{j=0}^m\sum_{k=0}^m\binom{m}{j}\binom{m+j}{j}\binom{m}{k}\binom{m+k}{k}(H_{m+k}-H_{m-k})\f{(-1)^{j+k}}{(j+k+1)\binom{j+k}{j}}\notag\\
&\quad=2p^2t\sum_{k=0}^m(-1)^k\binom{m}{k}\binom{m+k}{k}(H_{m+k}-H_{m-k})\f{\binom{k}{m}}{(k+1)\binom{m+k+1}{m}}\notag\\
&\quad=\f{2p^2t(-1)^m}{2m+1}H_{2m}\pmod{p^3}.
\end{align}
Substituting \eqref{sigma1decom1} and \eqref{sigma1decom2} into \eqref{sigma1decom}, we finally obtain the desired result.
\end{proof}

\begin{lemma}\label{sigma2}
For any odd prime $p$ and $p$-adic integer $x$ with $m<(p-1)/2$, we have
\begin{align*}
&p\sum_{j=0}^m\sum_{k=m+1}^{p-m-1}\binom{x}{j}\binom{x+j}{j}\binom{x}{k}\binom{x+k}{k}\sum_{n=0}^{j+k}\f{1}{n+1}\binom{n}{j}\binom{j}{n-k}\binom{p-1}{n}\\
&\quad\eq \f{pt(-1)^m}{2m+1}-\f{p^2t(-1)^m}{(2m+1)^2}-\f{p^2t^2(-1)^m}{(2m+1)^2}-\f{2p^2t(-1)^mH_{2m}}{2m+1}\pmod{p^3}.
\end{align*}
\end{lemma}

\begin{proof}
In view of Lemma \ref{xkx+1k}, for $0\leq j\leq m,\ m+1\leq k\leq p-m-1$, we have
\begin{align}\label{xjxk}
&\binom{x}{j}\binom{x+j}{j}\binom{x}{k}\binom{x+k}{k}\notag\\
&\quad\eq \f{pt(-1)^{m+k+1}\binom{m}{j}\binom{m+j}{j}\binom{m+k}{k}}{(k-m)\binom{k}{m}}(1+pt(H_{m+j}+H_{m+k}-H_{m-j}-H_{k-m-1}))\pmod{p^3}.
\end{align}
Note that
\begin{equation}\label{p-1n}
\binom{p-1}{n}=\prod_{s=1}^{n}\f{p-s}{s}=(-1)^n\prod_{s=1}^{n}\l(1-\f{p}{s}\r)\eq(-1)^n(1-pH_n)\pmod{p^2}
\end{equation}
for $0\leq n\leq p-1$. Combining \eqref{xjxk}, \eqref{p-1n} and Lemma \ref{identity1}, we deduce that
\begin{align}\label{sigma2decom}
&p\sum_{j=0}^m\sum_{k=m+1}^{p-m-1}\binom{x}{j}\binom{x+j}{j}\binom{x}{k}\binom{x+k}{k}\sum_{n=0}^{j+k}\f{1}{n+1}\binom{n}{j}\binom{j}{n-k}\binom{p-1}{n}\notag\\
&\quad\eq p^2t(-1)^{m+1}\sum_{j=0}^m\sum_{k=m+1}^{p-m-1}\f{(-1)^k\binom{m}{j}\binom{m+j}{j}\binom{m+k}{k}}{(k-m)\binom{k}{m}}(1+pt(H_{m+j}+H_{m+k}-H_{m-j}-H_{k-m-1}))\notag\\
&\qquad\times \sum_{n=0}^{j+k}\f{(-1)^n}{n+1}\binom{n}{j}\binom{j}{n-k}(1-pH_n)\notag\\
&\quad\eq p^2t(-1)^{m+1}\sum_{j=0}^m\sum_{k=m+1}^{p-m-1}\f{(-1)^k\binom{m}{j}\binom{m+j}{j}\binom{m+k}{k}}{(k-m)\binom{k}{m}}(1+pt(H_{m+j}+H_{m+k}-H_{m-j}-H_{k-m-1}))\notag\\
&\qquad\times \f{(-1)^{j+k}}{(j+k+1)\binom{j+k}{j}}\pmod{p^3},
\end{align}
where we have used the fact $H_n/(n+1)$ is a $p$-adic integer for each $n\in\{0,1,2,\ldots,p-1\}$. With the aid of Lemma \ref{identity3},
\begin{align}\label{sigma2decom1}
&p^2t(-1)^{m+1}\sum_{j=0}^m\sum_{k=m+1}^{p-m-1}\f{(-1)^k\binom{m}{j}\binom{m+j}{j}\binom{m+k}{k}}{(k-m)\binom{k}{m}}\f{(-1)^{j+k}}{(j+k+1)\binom{j+k}{j}}\notag\\
&\quad=p^2t(-1)^{m+1}\sum_{k=m+1}^{p-m-1}\f{\binom{m+k}{k}}{(k-m)\binom{k}{m}}\f{\binom{k}{m}}{(k+1)\binom{m+k+1}{m}}\notag\\
&\quad=p^2t(-1)^{m+1}\sum_{k=m+1}^{p-m-1}\f{1}{(k-m)(m+k+1)}\notag\\
&\quad=\f{p^2t(-1)^{m+1}}{2m+1}\sum_{k=m+1}^{p-m-1}\l(\f{1}{k-m}-\f{1}{m+k+1}\r).
\end{align}
Observe that
\begin{align*}
&p^2\sum_{k=m+1}^{p-m-1}\l(\f{1}{k-m}-\f{1}{m+k+1}\r)\\
&\quad=p^2\sum_{k=1}^{p-2m-1}\f{1}{k}-p-p^2\sum_{k=m+1}^{p-m-2}\f{1}{m+p-1-k+1}\\
&\quad\eq p^2\sum_{k=1}^{p-2m-1}\f{1}{k}-p+p^2\sum_{k=1}^{p-2m-2}\f{1}{k}\\
&\quad\eq 2p^2\sum_{k=1}^{p-2m-1}\f{1}{k}-p+\f{p^2}{2m+1}\\
&\quad\eq 2p^2H_{2m}-p+\f{p^2}{2m+1}\pmod{p^3}.
\end{align*}
Substituting this into \eqref{sigma2decom1}, we have
\begin{align}\label{sigma2decom2}
&p^2t(-1)^{m+1}\sum_{j=0}^m\sum_{k=m+1}^{p-m-1}\f{\binom{m}{j}\binom{m+j}{j}\binom{m+k}{k}}{(k-m)\binom{k}{m}}\f{(-1)^{j+k}}{(j+k+1)\binom{j+k}{j}}\notag\\
&\quad \eq \f{pt(-1)^m}{2m+1}-\f{p^2t(-1)^m}{(2m+1)^2}-\f{2p^2t(-1)^mH_{2m}}{2m+1}\pmod{p^3}.
\end{align}
On the other hand,
\begin{align}\label{sigma2decom3}
&p^3t^2(-1)^{m+1}\sum_{j=0}^m\sum_{k=m+1}^{p-m-1}\f{(-1)^k\binom{m}{j}\binom{m+j}{j}\binom{m+k}{k}}{(k-m)\binom{k}{m}}(H_{m+j}+H_{m+k}-H_{m-j}-H_{k-m-1})\f{(-1)^{j+k}}{(j+k+1)\binom{j+k}{j}}\notag\\
&\quad\eq -p^2t^2\f{\binom{2m}{m}\binom{p-1}{m}}{(p-2m-1)\binom{p-m-1}{m}\binom{p-1}{m}}(H_{2m}+H_{p-1}-H_{p-2m-2})\notag\\
&\quad\eq -\f{p^2t^2(-1)^m}{(2m+1)^2}\pmod{p^3}.
\end{align}
Substituting \eqref{sigma2decom2} and \eqref{sigma2decom3} into \eqref{sigma2decom}, we arrive at
\begin{align*}
&p\sum_{j=0}^m\sum_{k=m+1}^{p-m-1}\binom{x}{j}\binom{x+j}{j}\binom{x}{k}\binom{x+k}{k}\sum_{n=0}^{j+k}\f{1}{n+1}\binom{n}{j}\binom{j}{n-k}\binom{p-1}{n}\\
&\quad\eq \f{pt(-1)^m}{2m+1}-\f{p^2t(-1)^m}{(2m+1)^2}-\f{p^2t^2(-1)^m}{(2m+1)^2}-\f{2p^2t(-1)^mH_{2m}}{2m+1}\pmod{p^3}.
\end{align*}
This concludes the proof.
\end{proof}

\begin{lemma}\label{sigma3}
For any odd prime $p$ and $p$-adic integer $x$ with $m<(p-1)/2$, we have
\begin{align*}
&p\sum_{j=0}^m\sum_{k=p-m}^{p-1}\binom{x}{j}\binom{x+j}{j}\binom{x}{k}\binom{x+k}{k}\sum_{n=0}^{j+k}\f{1}{n+1}\binom{n}{j}\binom{j}{n-k}\binom{p-1}{n}\\
&\quad\eq \f{p^2t(t+1)(-1)^mH_{2m}}{2m+1}\pmod{p^3}.
\end{align*}
\end{lemma}

\begin{proof}
By Lemma \ref{xkx+1k}, for $0\leq j\leq m$ and $p-m \leq k\leq p-1$ we have
\begin{align*}
\binom{x}{j}\binom{x+j}{j}\binom{x}{k}\binom{x+k}{k}\eq\f{(-1)^mp^2t(t+1)\binom{m}{j}\binom{m+j}{j}}{k(k-m)\binom{m}{p-k}\binom{k}{m}}\pmod{p^3}.
\end{align*}
Therefore,
\begin{align*}
&p\sum_{j=0}^m\sum_{k=p-m}^{p-1}\binom{x}{j}\binom{x+j}{j}\binom{x}{k}\binom{x+k}{k}\sum_{n=0}^{j+k}\f{1}{n+1}\binom{n}{j}\binom{j}{n-k}\binom{p-1}{n}\\
&\quad\eq p^2t(t+1)(-1)^m\sum_{k=p-m}^{p-1}\sum_{j=p-1-k}^m\f{\binom{m}{j}\binom{m+j}{j}}{k(k-m)\binom{m}{p-k}\binom{k}{m}}\binom{p-1}{j}\binom{j}{p-1-k}\\
&\quad=p^2t(t+1)\sum_{k=p-m}^{p-1}\f{\binom{m}{p-1-k}\binom{m+p-1-k}{m}}{k(k-m)\binom{m}{p-k}\binom{k}{m}}\\
&\quad\eq -p^2t(t+1)(-1)^m\sum_{k=p-m}^{p-1}\f{1}{(m+k-p+1)(k-m)}\pmod{p^3},
\end{align*}
where we have used the fact
\begin{align*}
&\sum_{j=p-1-k}^{m}\binom{m}{j}\binom{-m-1}{j}\binom{j}{p-1-k}\\
&\quad=\binom{m}{p-1-k}\sum_{j=p-1-k}^{m}\binom{-m-1}{j}\binom{k+m+1-p}{m-j}\\
&\quad=\binom{m}{p-1-k}\binom{k-p}{m}\\
&\quad=(-1)^m\binom{m}{p-1-k}\binom{m+p-1-k}{m}
\end{align*}
by the Chu-Vandermonde identity. Applying the partial fraction decomposition, we get
\begin{align*}
&\sum_{k=p-m}^{p-1}\f{1}{(m+k-p+1)(k-m)}\\
&\quad\eq \f{1}{2m+1}\l(\sum_{k=p-m}^{p-1}\f{1}{k-m}-\sum_{k=p-m}^{p-1}\f{1}{k+1+m-p}\r)\\
&\quad=\f{1}{2m+1}\l(\sum_{k=1}^{m}\f{1}{p-k-m}-\sum_{k=1}^{m}\f{1}{k}\r)\\
&\quad\eq-\f{H_{2m}}{2m+1}\pmod{p}.
\end{align*}
Combining the above, we obtain the desired result.
\end{proof}

\begin{lemma}\label{sigma5}
For any odd prime $p$ and $p$-adic integer $x$ with $m<(p-1)/2$, we have
\begin{align*}
&p\sum_{j=m+1}^{p-m-1}\sum_{k=m+1}^{p-m-1}\binom{x}{j}\binom{x+j}{j}\binom{x}{k}\binom{x+k}{k}\sum_{n=0}^{j+k}\f{1}{n+1}\binom{n}{j}\binom{j}{n-k}\binom{p-1}{n}\\
&\quad\eq-\f{2p^2t^2(-1)^m}{(2m+1)^2}-\f{2p^2t^2(-1)^m H_{2m}}{2m+1}\pmod{p^3}.
\end{align*}
\end{lemma}

\begin{proof}
By Lemma \ref{xkx+1k}, for $m+1\leq j,k\leq p-1-m$, we have
\begin{align*}
&\binom{x}{j}\binom{x+j}{j}\binom{x}{k}\binom{x+k}{k}\\
&\quad\eq \f{(-1)^{j+k}p^2t^2\binom{m+j}{j}\binom{m+k}{k}}{(j-m)(k-m)\binom{j}{m}\binom{k}{m}}\pmod{p^3}.
\end{align*}
Hence we get
\begin{align}\label{sigma5decom}
&p\sum_{j=m+1}^{p-m-1}\sum_{k=m+1}^{p-m-1}\binom{x}{j}\binom{x+j}{j}\binom{x}{k}\binom{x+k}{k}\sum_{n=0}^{j+k}\f{1}{n+1}\binom{n}{j}\binom{j}{n-k}\binom{p-1}{n}\notag\\
&\quad\eq p^2t^2\sum_{j=m+1}^{p-m-1}\sum_{k=\max\{m+1,p-1-j\}}^{p-m-1}\f{(-1)^{j}\binom{m+j}{j}\binom{m+k}{k}\binom{k}{p-1-j}}{(j-m)(k-m)\binom{j}{m}\binom{k}{m}}\notag\\
&\quad\eq-\f{p^2t^2(-1)^m}{(2m+1)\binom{2m}{m}}\sum_{k=m+1}^{p-m-1}\f{\binom{m+k}{k}}{k-m}+p^2t^2\sum_{j=m+1}^{p-m-2}\sum_{k=p-1-j}^{p-m-1}\f{(-1)^{j}\binom{m+j}{j}\binom{m+k}{k}\binom{k}{p-1-j}}{(j-m)(k-m)\binom{j}{m}\binom{k}{m}}\notag\\
&\quad:=\Sigma_1+\Sigma_2\pmod{p^3}.
\end{align}

From \cite[(1.132)]{G} we know
\begin{equation}\label{Gouldlist}
\sum_{k=M}^{N-1}\f{\binom{k-1}{M-1}}{N-k}=\binom{N-1}{M-1}\sum_{k=M}^{N-1}\f1k.
\end{equation}
Putting $M=m+1,\ N=p-m$ in \eqref{Gouldlist}, we get
\begin{equation}\label{Gouldlistsubs}
\sum_{k=m+1}^{p-m-1}\f{\binom{k-1}{m}}{p-m-k}=\binom{p-m-1}{m}(H_{p-m-1}-H_m).
\end{equation}
Replacing $k$ with $p-k$ in \eqref{Gouldlistsubs}, we obtain
$$
\sum_{k=m+1}^{p-m-1}\f{\binom{p-k-1}{m}}{k-m}=\binom{p-m-1}{m}(H_{p-m-1}-H_m).
$$
Then, modulo $p$, we have
$$
(-1)^m\sum_{k=m+1}^{p-m-1}\f{\binom{m+k}{k}}{k-m}\eq0\pmod{p},
$$
which implies that
\begin{equation}\label{Sigma1}
\Sigma_1\eq0\pmod{p^3}.
\end{equation}

Now we consider $\Sigma_2$. Note that for $m+1\leq j\leq p-m-2$,
\begin{align}\label{Sigma2sum}
\sum_{k=p-1-j}^{p-m-1}\f{\binom{m+k}{k}\binom{k}{p-1-j}}{(k-m)\binom{k}{m}}&\eq \sum_{k=p-1-j}^{p-m-1}\f{(-1)^k\binom{p-m-1}{k}\binom{k}{p-1-j}}{(k-m)\binom{k}{m}}\notag\\
&=\f{\binom{p-m-1}{j-m}}{m+1}\sum_{k=p-1-j}^{p-m-1}\f{(-1)^k\binom{j-m}{j+k-p+1}}{\binom{k}{m+1}}\notag\\
&=\f{(-1)^j\binom{p-m-1}{j-m}}{m+1}\sum_{k=0}^{j-m}\f{(-1)^k\binom{j-m}{k}}{\binom{p-1-j+k}{m+1}}\pmod{p}.
\end{align}
From \cite[(4.2)]{G}, we know for $b\geq c>0$,
\begin{equation}\label{Gouldlist'}
\sum_{k=0}^{N}\f{(-1)^k\binom{N}{k}}{\binom{b+k}{c}}=\f{c}{(N+c)\binom{N+b}{b-c}}.
\end{equation}
Substituting \eqref{Gouldlist'} with $N=j-m,\ b=p-1-j,\ c=m+1$ into \eqref{Sigma1}, we arrive at
\begin{equation}\label{Sigma2sum'}
\sum_{k=p-1-j}^{p-m-1}\f{\binom{m+k}{k}\binom{k}{p-1-j}}{(k-m)\binom{k}{m}}\eq \f{(-1)^j\binom{p-m-1}{j-m}}{(j+1)\binom{p-m-1}{j+1}}\pmod{p}.
\end{equation}
Therefore,
\begin{align*}
\Sigma_2&=p^2t^2\sum_{j=m+1}^{p-m-2}\sum_{k=p-1-j}^{p-m-1}\f{(-1)^{j}\binom{m+j}{j}\binom{m+k}{k}\binom{k}{p-1-j}}{(j-m)(k-m)\binom{j}{m}\binom{k}{m}}\\
&\eq p^2t^2\sum_{j=m+1}^{p-m-2}\f{\binom{m+j}{j}\binom{p-m-1}{j-m}}{(j-m)(j+1)\binom{p-m-1}{j+1}\binom{j}{m}}\\
&\eq -p^2t^2(-1)^m\sum_{j=m+1}^{p-m-2}\f{1}{(j-m)(j+m+1)}\\
&=\f{p^2t^2(-1)^m}{2m+1}\sum_{j=m+1}^{p-m-2}\l(\f{1}{j+m+1}-\f{1}{j-m}\r)\\
&=\f{p^2t^2(-1)^m}{2m+1}(H_{p-1}-H_{2m+1}-H_{p-2m-2})\pmod{p^3}.
\end{align*}
Then the desired result follows from the fact $H_{p-1}\eq0\pmod{p}$ and $H_{p-1-k}\eq H_k\pmod{p}$.
\end{proof}

\section{Proof of Theorem \ref{mainth}}

We divide the proof into three cases.

{\it Case 1}. $m<(p-1)/2$.

From \cite[(3.5)]{Liu2018}, we know
$$
\sum_{n=0}^{p-1}\binom{n}{j}\binom{n}{k}=p\sum_{n=0}^{j+k}\f{1}{n+1}\binom{n}{j}\binom{j}{n-k}\binom{p-1}{n}.
$$
Therefore,
\begin{align*}
\sum_{n=0}^{p-1}s_n(x)^2&=\sum_{n=0}^{p-1}\sum_{j=0}^n\sum_{k=0}^n\binom{n}{j}\binom{n}{k}\binom{x}{j}\binom{x+j}{j}\binom{x}{k}\binom{x+k}{k}\\
&=\sum_{j=0}^{p-1}\sum_{k=0}^{p-1}\binom{x}{j}\binom{x+j}{j}\binom{x}{k}\binom{x+k}{k}\sum_{n=0}^{p-1}\binom{n}{j}\binom{n}{k}\\
&=p\sum_{j=0}^{p-1}\sum_{k=0}^{p-1}\binom{x}{j}\binom{x+j}{j}\binom{x}{k}\binom{x+k}{k}\sum_{n=0}^{j+k}\f{1}{n+1}\binom{n}{j}\binom{j}{n-k}\binom{p-1}{n}\\
&=\sum_{s=1}^9\sigma_s,
\end{align*}
where
\begin{gather*}
\sigma_1=\sum_{j=0}^m\sum_{k=0}^mf(j,k),\quad \sigma_2=\sum_{j=0}^m\sum_{k=m+1}^{p-1-m}f(j,k),\quad \sigma_3=\sum_{j=0}^m\sum_{k=p-m}^{p-1}f(j,k),\\
\sigma_4=\sum_{j=m+1}^{p-1-m}\sum_{k=0}^mf(j,k),\quad \sigma_5=\sum_{j=m+1}^{p-1-m}\sum_{k=m+1}^{p-1-m}f(j,k),\quad \sigma_6=\sum_{j=m+1}^{p-1-m}\sum_{k=p-m}^{p-1}f(j,k),\\
\sigma_7=\sum_{j=p-m}^{p-1}\sum_{k=0}^mf(j,k),\quad \sigma_8=\sum_{j=p-m}^{p-1}\sum_{k=m+1}^{p-1-m}f(j,k),\quad \sigma_9=\sum_{j=p-m}^{p-1}\sum_{k=p-m}^{p-1}f(j,k),
\end{gather*}
and $f(j,k)\ (0\leq j,k\leq p-1)$ stand for the summands
$$
p\binom{x}{j}\binom{x+j}{j}\binom{x}{k}\binom{x+k}{k}\sum_{n=0}^{j+k}\f{1}{n+1}\binom{n}{j}\binom{j}{n-k}\binom{p-1}{n}.
$$
By the symmetry of $j$ and $k$, we have $\sigma_2=\sigma_4,\ \sigma_3=\sigma_7$. Meanwhile, in view of Lemma \ref{xkx+1k}, we immediately obtain $\sigma_6\eq\sigma_8\eq\sigma_9\eq0\pmod{p^3}$. Therefore,
\begin{equation}\label{keystep}
\sum_{n=0}^{p-1}s_n(x)^2\eq \sigma_1+2\sigma_2+2\sigma_3+\sigma_5\pmod{p^3}.
\end{equation}
This, together with Lemmas \ref{sigma1}--\ref{sigma5}, gives
\begin{equation}\label{keystep'}
\sum_{n=0}^{p-1}s_n(x)^2\eq \f{p(-1)^m}{2m+1}+\f{2p t(-1)^m}{2m+1}-\f{2p^2t(-1)^m}{(2m+1)^2}-\f{4p^2t^2(-1)^m}{(2m+1)^2}\pmod{p^3}.
\end{equation}

On the other hand, modulo $p^3$, the right-hand side of \eqref{sunconjeq} is equivalent to
\begin{align}\label{keystep''}
(-1)^m\f{p+2p t}{2m+1+2p t}&\eq (-1)^m\f{(p+2p t)(2m+1-2p t)}{(2m+1)^2}\notag\\
&\eq \f{p(-1)^m}{2m+1}+\f{2p t(-1)^m}{2m+1}-\f{2p^2t(-1)^m}{(2m+1)^2}-\f{4p^2t^2(-1)^m}{(2m+1)^2}.
\end{align}
Combining \eqref{keystep'} and \eqref{keystep''}, we conclude the proof of Theorem \ref{mainth} in this case.

\medskip

{\it Case 2}. $m>(p-1)/2$.

Clearly, $s_n(x)=s_n(-1-x)$ for all nonnegative integer $n$. So
$$
\sum_{n=0}^{p-1}s_n(x)^2=\sum_{n=0}^{p-1}s_n(-1-x)^2.
$$
Since $\<-1-x\>_{p}=p-1-m<(p-1)/2$, by Theorem \ref{mainth} in Case 1, we have
\begin{align*}
\sum_{n=0}^{p-1}s_n(x)^2&\eq (-1)^{\<-1-x\>_{p}}\f{p+2(-1-x-\<-1-x\>_{p})}{2(-x-1)+1}\\
&=(-1)^{\<x\>_p}\f{p+2(x-\<x\>_p)}{2x+1}\pmod{p^3}.
\end{align*}
This proves Theorem \ref{mainth} in Case 2.

\medskip

{\it Case 3}. $m=(p-1)/2$ and $x\neq -1/2$.

In this case, $m+1>p-1-m$. So we have
\begin{equation}\label{decom}
\sum_{n=0}^{p-1}s_n(x)^2\eq \sigma_1+2\sigma_3\pmod{p^3},
\end{equation}
where $\sigma_1$ and $\sigma_3$ are defined as in Case 1.

We first evaluate $\sigma_1$ modulo $p^3$. Write $x=-1/2+ps=m+pt$. It is easy to see that for $0\leq k\leq m$ we have
\begin{align*}
\binom{x}{k}\binom{x+k}{k}&=\prod_{i=0}^{k-1}\l(-\f12-i+ps\r)\l(\f12+i+ps\r)\bigg/ k!^2\\
&= (-1)^k\prod_{i=0}^{k-1}\l(\l(\f12+i\r)^2-p^2s^2\r)\bigg/k!^2\\
&\eq (-1)^k\prod_{i=0}^{k-1}\l(\f12+i\r)^2\bigg/k!^2\l(1-p^2s^2\sum_{j=0}^{k-1}\f{1}{(1/2+j)^2}\r)\\
&=(-1)^k\binom{-1/2}{k}^2\l(1-4p^2s^2\sum_{j=1}^{k}\f{1}{(2j-1)^2}\r)\pmod{p^3}.
\end{align*}
Combining this with \cite[(3.6)]{Liu2018}, we deduce that
\begin{align}\label{xkx+1km}
\binom{x}{k}\binom{x+k}{k}&\eq \binom{m}{k}\binom{m+k}{k}\l(1+p^2(1-4s^2)\sum_{j=1}^{k}\f{1}{(2j-1)^2}\r)\notag\\
&=\binom{m}{k}\binom{m+k}{k}\l(1-4p^2t(t+1)\sum_{j=1}^{k}\f{1}{(2j-1)^2}\r)\pmod{p^3},
\end{align}
where we have used $s=t+1/2$. Therefore, by the symmetry of $j$ and $k$,
\begin{align}\label{msigma1}
\sigma_1&\eq p\sum_{j=0}^m\sum_{k=0}^m\binom{m}{j}\binom{m+j}{j}\binom{m}{k}\binom{m+k}{k}\l(1-8p^2t(t+1)\sum_{i=1}^{j}\f{1}{(2i-1)^2}\r)\notag\\
&\quad\times \sum_{n=0}^{j+k}\f{1}{n+1}\binom{n}{j}\binom{j}{n-k}\binom{p-1}{n}\pmod{p^3}.
\end{align}
By Lemma \ref{identity2} with $M=m$ and $y=p$, we have
\begin{align}\label{msigma1key1}
&p\sum_{j=0}^m\sum_{k=0}^m\sum_{n=0}^{j+k}\f{1}{n+1}\binom{m}{j}\binom{m+j}{j}\binom{m}{k}\binom{m+k}{k}\binom{n}{j}\binom{j}{n-k}\binom{p-1}{n}\notag\\
&\quad\eq(-1)^m(1-4p^2H_{m}^{(2)})\notag\\
&\quad\eq (-1)^m\pmod{p^3},
\end{align}
where we have used the fact that $H_{(p-1)/2}^{(2)}\eq0\pmod{p}$ for $p>3$. Moreover, since $H_{p-1}^{(2)}\eq0\pmod{p}$ for $p>3$, we have
\begin{align}\label{msigma1key2}
&8p^3t(t+1)\sum_{j=0}^m\sum_{k=0}^m\binom{m}{j}\binom{m+j}{j}\binom{m}{k}\binom{m+k}{k}\sum_{i=1}^j\f{1}{(2i-1)^2}\sum_{n=0}^{j+k}\f{1}{n+1}\binom{n}{j}\binom{j}{n-k}\binom{p-1}{n}\notag\\
&\quad\eq 8p^2t^2\binom{2m}{m}^3\l(H_{2m}^{(2)}-\f14H_m^{(2)}\r)\notag\\
&\quad\eq0\pmod{p^3}.
\end{align}
Substituting \eqref{msigma1key1} and \eqref{msigma1key2} into \eqref{msigma1}, we obtain
\begin{equation}\label{msigma1reduce}
\sigma_1\eq (-1)^m\pmod{p^3}.
\end{equation}

Now we evaluate $\sigma_2$ modulo $p^3$. In view of Lemma \ref{xkx+1k} and \eqref{xkx+1km},
$$
\sigma_3\eq p^3t(t+1)(-1)^m\sum_{j=0}^m\sum_{k=m+1}^{p-1}\f{\binom{m}{j}\binom{m+j}{j}}{k(k-m)\binom{m}{p-k}\binom{k}{m}}\sum_{n=0}^{j+k}\f{1}{n+1}\binom{n}{j}\binom{j}{n-k}\binom{p-1}{n}\pmod{p^3}.
$$
Then, by the same argument as in the proof of Lemma \ref{sigma3}, we obtain
\begin{equation}\label{msigma3reduce}
\sigma_3\eq -p^2t(t+1)(-1)^m\sum_{k=m+1}^{p-1}\f{1}{(k-m)^2}=-p^2t(t+1)(-1)^mH_{p-1-m}^{(2)}\eq0\pmod{p^3}.
\end{equation}
Substituting \eqref{msigma1reduce} and \eqref{msigma3reduce} into \eqref{decom}, we arrive at
$$
\sum_{n=0}^{p-1}s_n(x)^2\eq (-1)^m\pmod{p^3}.
$$
Furthermore, it is routine to verify that
$$
(-1)^m\f{p+2p t}{2m+1+2p t}=(-1)^m.
$$
Therefore, Theorem \ref{mainth} in Case 3 holds.

The proof of Theorem \ref{mainth} is now complete.\qed


\begin{thebibliography}{99}
\bibitem{G} H.W. Gould, Combinatorial Identities, Morgantown Printing and Binding Co., West Virginia, 1972.

\bibitem{Guo} V.J.W. Guo, Proof of Sun's conjectures on integer-valued polynomials, J. Math. Anal. Appl. 444 (2016), 182--191.

\bibitem{KW} K. Kimoto and M. Wakayama, Ap\'ery-like numbers arising from special values of spectral zeta function for non-commutative harmonic oscillators, Kyushu J. Math. 60 (2006), 383--404.

\bibitem{Liu2017} J.-C. Liu, Proof of some divisibility results on sums involving binomial coefficients, J. Number Theory 180 (2017), 566--572.

\bibitem{Liu2018} J.-C. Liu, A generalized supercongruence of Kimoto and Wakayama, J. Math. Anal. Appl. 467 (2018), 15--25.

\bibitem{LOS} L. Long, R. Osburn and H. Swisher, On a conjecture of Kimoto and Wakayama, Proc. Amer. Math. Soc. 144 (2016), 4319--4327.

\bibitem{Mao} G.-S. Mao, Proof of some conjectural congruences involving binomial coefficients and Ap\'ery-like numbers, Front. Math. (2025). https://doi.org/10.1007/s11464-024-0146-x

\bibitem{Mortenson} E. Mortenson, A supercongruence conjecture of Rodriguez-Villegas for a certain truncated hypergeometric function, J. Number Theory 99 (2003), 139--147.

\bibitem{RV}  F. Rodriguez-Villegas, Hypergeometric families of Calabi-Yau manifolds, in: Calabi-Yau Varieties and Mirror Symmetry (Toronto, ON, 2001), Fields Inst. Commun., 38, Amer. Math. Soc., Providence, RI, 2003, pp. 223--231.

\bibitem{SunZH2014} Z.-H. Sun, Generalized Legendre polynomials and related supercongruences, J. Number Theory 143 (2014), 293--319.

\bibitem{SunZH2022} Z.-H. Sun, Congruences for certain families of Ap\'ery-like sequences, Czech. Math. J. 72 (2022), 875--912.

\bibitem{SunZW2012} Z.-W. Sun, A refinement of a congruence result by van Hamme and Mortenson, Illinois J. Math. 56 (2012), 967--979.

\bibitem{SunZW2013} Z.-W. Sun, Supercongruences involving products of two binomial coefficients, Finite Fields Appl. 22 (2013), 24--44.

\bibitem{SunZW2017} Z.-W. Sun, Supercongruences involving dual sequences, Finite Fields Appl. 46 (2017), 179--216.

\bibitem{WangZhong} R.-H. Wang and M.X.X. Zhong, Congruences related to dual sequences and Catalan numbers, European J. Combin. 101 (2022), 103458.
\end{thebibliography}
\end{document}